\documentclass[12pt, letterpaper]{article}

\usepackage[affil-it]{authblk}
\usepackage[utf8x]{inputenc}
\usepackage[T1]{fontenc}
\usepackage[pdftex]{graphicx}
\usepackage{float, color}
\usepackage{authblk}
\usepackage{ucs}
\usepackage{amsmath}
\usepackage{amsfonts}
\usepackage{amssymb}
\usepackage{amsthm}

\usepackage[left=2.8cm,right=2.8 cm,top=2cm,bottom=2cm]{geometry}

\newtheorem{lem}{Lemma}
\newtheorem{thm}{Theorem}
\newtheorem{atheo}{Theorem}

\newtheorem{corl}{Corollary}
\newtheorem{defin}{Definition}
\newtheorem{proposition}{Proposition}
\newtheorem{rem}{Remark}

\newcommand{\R}{\mathbb{R}}

\newcommand{\eps}{\varepsilon}
\newcommand{\ds}{\displaystyle}

\newcommand{\sign}{\operatorname{sign}}
\usepackage{pdfpages}
\usepackage{hyperref}

\hypersetup{
    unicode=true,          
    pdftoolbar=true,        
    pdfmenubar=true,        
    pdffitwindow=false,     
    pdfstartview={FitH},    
    pdftitle={},    
    pdfauthor={M El Smaily R L Jerrard},     
    pdfsubject={},   
    pdfcreator={},   
    pdfproducer={Producer}, 
    pdfkeywords={semilinear wave equation} {minimal timelike hyper-surfaces, wave equation, Ginzburg Landau nonlinearity} {key3}, 
    pdfnewwindow=true,      
    colorlinks=true,       
    linkcolor=black,          
    citecolor=red,        
    filecolor=magenta,      
    urlcolor=blue           
}

\begin{document}
\title{A refined description of evolving interfaces in certain nonlinear wave equations}

\author{Mohammad El Smaily \footnote{\textsc{Department of Mathematics \& Statistics, University of New Brunswick, Fredericton, NB, E3B5A3, Canada.} \hskip 5 pt
  \textit{E-mail address}: \texttt{m.elsmaily@unb.ca}} ~~ ~~ Robert L. Jerrard \footnote{\textsc{Department of Mathematics, University of Toronto, Toronto, Ontario, M5S2E4, Canada.}
  \hskip 5 pt \textit{E-mail address:} \texttt{rjerrard@math.toronto.edu}}}

\date{August 19, 2017}

\maketitle
\abstract{We improve on recent results that establish the existence of
solutions of certain semilinear wave equations possessing an interface
that roughly sweeps out a timelike surface of vanishing mean curvature
in Minkowski space. Compared to earlier work, we  present sharper estimates,
in stronger norms, of the solutions in question. }
\section{Introduction}

The goal of this paper is to refine
recent work  \cite{galjer, jer0909} that
proves the existence of a solution $u=u_{\eps}(t,x)$ ($x\in\R^{n}$, $t\in \R$)  of the semilinear wave equation 
\begin{equation}\label{nlw}
u_{tt}-\Delta u +\frac{2}{\eps^{2}}(u^{2}-1)u=0,\qquad\quad 0 < \eps \ll 1\mbox{ fixed}
\end{equation}
such that, roughly speaking, $u$ exhibits an interface near
a timelike hypersurface
whose Minkowskian mean curvature identically vanishes, as long as the 
hypersurface remains smooth.
To describe the problem, let $\Gamma$ be a smooth timelike embedded hypersurface in $(-T_*,T^*)\times \R^n$, for some $T_*,T^*>0$
of vanishing Minkowski mean curvature, and such that 
\[
\Gamma_t  := \{ x\in \R^n : (t,x)\in \Gamma\}
\]
is homeomomorphic to $\mathbb{S}^{n-1}$ for every $t$. The condition that the Minkowskian mean curvature vanishes
is a nonlinear geometric wave equation, and smooth
solutions are known to exist, locally in  $t$, for suitable compact Cauchy data,
see for example  \cite{mil}. 
We remark that when $n=2$ (which we
will assume throughout most of this paper) the equation
is in some sense integrable and there is essentially an
explicit formula for solutions (see Section \ref{background1} below).

The fact that $\Gamma_t$ is a topological sphere for every $t$ implies that 
$(-T_*, T^*) \times \R^n) \setminus \Gamma$ consists of two components,
one bounded and one unbounded. Let $\mathcal O$ denote the bounded component, and
\[
\sign_{\mathcal O}(t,x) := \begin{cases}
1&\mbox{ if }(t,x)\in {\mathcal O}\\
-1&\mbox{ if }(t,x)\in {\mathcal O}^{c}.
\end{cases}
\]

The following result was proved in \cite{galjer, jer0909}:

\begin{atheo}[\cite{galjer, jer0909}]\label{TheoremA}
Given $\Gamma$ as above, for every $\eps\in (0,1]$
there exists a  solution $u$ of (\ref{nlw}) such that 
for any $T_0 < T_*$ and $T^0<T^*$,
 \begin{equation}\label{u_weakly_close_toA}
 \left\Vert u_\eps- \sign_{\mathcal O}\right\Vert_{L^{2}((- T_0,T^0)\times\R^{n})}\leq C\sqrt{\eps}.
 \end{equation}
 for a  constant $C$ that may depend on $T_0,T^0$ but is independent of $\eps$.
\end{atheo}

  In \cite{jer0909}, Theorem \ref{TheoremA} is proved under the assumption that $\Gamma_0$ is a topological
torus, but allowing rather general initial velocity for $\Gamma$, whereas the proof in \cite{galjer}
allows $\Gamma_0$ to be an arbitrary smooth connected compact manifold with zero initial velocity. 
The theorem as stated above follows by combining arguments from the two papers \cite{galjer} and \cite{jer0909}. For $n=2$ and $\Gamma_0$ homeomorphic to $\mathbb S^1$,  which is our main focus, 
it follows directly from \cite{jer0909}.

Our goal is to give a more precise description of the solution $u_\eps$
found in Theorem \ref{TheoremA}. In particular, heuristic arguments
suggest that it should satisfy
\begin{equation}\label{profile1}
u_\eps(t, x) \approx q( \frac{\tilde d(t,x)}\eps), \qquad\qquad q(s) := \tanh(s)
\end{equation}
where 
 $\tilde d(\cdot, \cdot)$
is (a small perturbation of) the signed Minkowskian distance from $\Gamma$ (see 
\eqref{sd} below for a definition).
The profile  $q = \tanh$ arises naturally from the fact that  it satisfies $-q'' + 2(q^2-1)q=0$, making it
a stationary solution of \eqref{nlw} in $1$ dimension with $\eps = 1$.

The estimates in \cite{galjer, jer0909} are however too weak to provide a convincing demonstration of 
\eqref{profile1}, since\footnote{We are being a little imprecise here, since the signed Minkowskian  distance function $d$
is only defined near $\Gamma$. So to state this estimate properly,
one would need either to restrict attention to this neighborhood, or extend $d$
in some way to the complement of this neighborhood.}

\[
\left\Vert \sign_{\mathcal O} - q(\frac d \eps)
\right\Vert_{L^{2}((-T_0,T^0)\times\R^{n})} \approx \sqrt{\eps}.
\]
Thus \eqref{u_weakly_close_toA} implies\footnote{
Indeed, the main results  \cite{jer0909}  are stated in this way,
that is, with an estimate of $u_\eps - q(d/\eps)$ rather than
$u_\eps - \sign_{\mathcal O}$. This is correct
but arguably misleading.} that
$\left\Vert u_\eps - q(\frac d \eps)
\right\Vert_{L^{2}((-T_0,T^0)\times\R^{n})} \lesssim \sqrt{\eps}$,
but at the same time, the  scaling in \eqref{u_weakly_close_toA}
means the estimate
is too weak to determine whether $u_\eps$ is closer to
$\sign_{\mathcal O}$ or $q(d/\eps)$ or indeed some other profile.

In our main result, we restrict our attention
to $n=2$, and we establish a more precise description
of the solution $u_\eps$ from Theorem \ref{TheoremA}.
In our main theorem, we consider a solution $u_\eps$
of \eqref{nlw} as constructed in \cite{galjer, jer0909}, and we prove that
\begin{equation}\label{preview}
\left\Vert u_\eps - U_\eps
\right\Vert_{L^{2}((-T_0,T^0)\times\R^{2})} \le C\eps^{3/2},
\qquad 
\quad
\left\Vert D(u_\eps - U_\eps)
\right\Vert_{L^{2}((-T_0,T^0)\times\R^{2})} \le C\eps^{1/2}
\end{equation}
for some function $U_\eps$, constructed below, of approximately the form
$U_\eps = q(\tilde d/\eps)$, where $\tilde d$ is a perturbation
of the signed Minkowskian distance to $\Gamma$. 
This improves on \eqref{u_weakly_close_toA} in that
we have both a stronger norm 
and stronger estimates. 
In particular, as will be apparent from the construction
of $U_\eps$ below, conclusion \eqref{preview} 
may be understood as a precise and satisfactory formulation of the heuristic principle
\eqref{profile1}.
Our construction and our results will show that 
$\| Du_\eps\|_{L^2}, \|DU_\eps\|_{L^2}$ both diverge at a rate of $\eps^{-1/2}$
as $\eps\to 0$, which makes the second estimate in \eqref{preview} 
quite striking.

The  construction of $U_\eps$ and the full statement of our main theorem are presented in Section \ref{constructU}
below. 

Our arguments in this paper do not directly address
the wave equation \eqref{nlw}. Instead, we start from estimates
proved in \cite{galjer, jer0909} and summarized 
in Section \ref{priorwork} below, which provide considerably more
information about the solution $u_\eps$ than is stated in 
Theorem \ref{TheoremA} above. 
We will prove our main results
by squeezing as much information as possible 
out of these prior estimates.

The results of \cite{galjer, jer0909}, on which we improve here, may be seen as Minkowskian analogs of 
the large body of theory that gives rigorous asymptotic 
descriptions of interfaces in  semilinear elliptic and parabolic equations
associated to a double-well potential, see for example
\cite{deMottoni, Pacard, dPKW}.
Prior to  \cite{galjer, jer0909},
the connection between \eqref{nlw} and timelike
extremal surfaces, as well as related questions, were explored by formal arguments in
\cite{neu, rn}, and in the cosmology literature,
see for example \cite{ Coleman77, kibble, vs}, in connection with
hypothetical cosmic domain walls.
Some conditional results in the direction of \cite{galjer, jer0909} were obtained
a little earlier in \cite{bno}, and
results about scattering of a smooth, nearly flat interface
in a solution of \eqref{nlw}
are proved in \cite{sc}, following earlier results about scattering
of nearly flat Minkowskian extremal hypersurfaces, see  \cite{Brendle, lind}.

\subsection{Normal coordinates and the signed distance function}\label{background1}

Most of our analysis will be carried out in Minkowsian normal
coordinates near $\Gamma$, which we now describe.

First, we will write $\psi: (-T_*,T^*)\times \mathbb S^1\to \R^{1+2}$
to denote a map that parametrizes the extremal surface 
$\Gamma$. We will write $(y_0, y_1)$ to denote
a generic point in $(-T_*,T^*)\times \mathbb S^1$, and we will 
take $\psi$ to have the form
\[
\psi(y_0, y_1) = (y_0, \vec \psi(y_0, y_1)),
\]
Although we  will not use this fact, we remark that
$\Gamma :=\mbox{ Image}(\Psi)$
is extremal (that is, has vanishing mean curvature)
if $\vec \psi: (-T_*,T^*)\times \mathbb S^1\to \R^2$ has the form
\begin{equation}\label{param}
\vec \psi(y_0,y_1) = \frac 12 (a(y_0+y_1) + b(y_0 - y_1))
\end{equation}
for some smooth $a,b:\mathbb S^1\to \R^2$ such that $|a'| = |b'| = 1$
everywhere; see \cite{BHNO}
for a discussion.

Next, for $(y_0, y_1)\in (-T_*,T^*)$ and $y_2\in \R$, we define
\begin{equation}\label{normalc}
\Psi(y_0, y_1, y_2) := \psi(y_0, y_1) + y_2 \nu(y_0,y_1) \in \R^{1+2},
\end{equation}
where $\nu(y_0,y_1)$ is the (Minkowskian) unit normal to $\Gamma$ at $\psi(y_0,y_1)$,
and we orient $\nu$ so that $\Psi(y_0,y_1, y_2)\in \mathcal O$ for $y_2>0$, where we recall that $\mathcal O$ is the bounded set
enclosed by $\Gamma$. Thus $\nu$ ``points inward".
We will restrict the domain of $\Psi$ to a set of the form
\begin{equation}\label{domPsi}
\mbox{Domain}(\Psi) = (-T_1,T^1)\times \mathbb S^1\times (-2\rho, 2\rho),
\end{equation}
for $T_1,T^1, \rho$ fixed in Proposition \ref{old.results} below. We also 
tacitly require  that
$\Psi$ is a diffeomorphism onto its image; for a given $T_1, T^1$, this
can always be achieved by shrinking $\rho$.
We will write $\mathcal N := \mbox{Image}(\Psi)\subset \R^{1+2}$
and for points $(t,x)\in \mathcal N$, we will use the change
of variables
\[
\mathcal N \ni (t,x) = \Psi(y_{0},y_{1},y_{2}).
\]
Equivalently, we can view  $(y_{0},y_{1},y_{2})$ as defining a local coordinate system
in $\mathcal N$.
We will sometimes refer to these as \emph{normal coordinates} near $\Gamma.$
The $y_2$ coordinate is exactly the signed
Minkowskian distance $d(\cdot, \cdot)$ to $\Gamma$, in the sense that for $(t,x)\in \mathcal N$,
\begin{equation}
(t,x) = \Psi(y_0,y_1, y_2) \qquad \Longleftrightarrow \qquad d(t,x) = y_2.
\label{sd}\end{equation}
One can take \eqref{sd} to be the definition of the signed distance.
Alternately, for a proof of \eqref{sd} that starts from an eikonal equation that characterizes the signed distance function,
see for example \cite{jer0909}, Proposition 5 and Corollary 7. 

\subsection{Main Theorem, and Construction of $U_\eps$}\label{constructU}

Given a solution $u_\eps$ of the semilinear wave equation \eqref{nlw}
on $\R^{1+2}$, we will always write $v_\eps:(-T_*,T^*)\times \mathbb S^1\times (-2\rho, 2\rho)\to \R$
to denote the same solution written in the Minkowskian normal coordinate system.
That is, we set
\begin{equation}\label{v.def}
v_\eps := u_\eps\circ\Psi.
\end{equation}

We will use the notation 
\[
q(z) = \tanh (z),\qquad \qquad q_\eps(z) =\tanh(\frac z\eps).
\]
Given $f:\R\to \R$ and $s\in \R$, we write $\tau_s f$ to denote the translation of $f$ by $s$:
\[
\tau_s f(z) := f(z-s).
\]
For $\rho$ to be fixed in Proposition \ref{old.results} below, we define
$Q_\eps:\R\to \R$ by
\begin{equation}\label{Qeps.def}
Q_\eps(z) := q_\eps(z) \chi(z) + (1-\chi(z)) \sign(z)
\end{equation}
where $\chi\in C^\infty(\R)$  is a fixed even, nonnegative function such that
\[
\chi(z) = 1 \mbox{ if }|z|\le \rho/3, \qquad \chi(z) = 0\mbox{ if }|z|\ge 2\rho/3, \qquad
z\chi'(z) \le 0.
\]
It is easy to see that for every $k\in \mathbb N$,  there exist constants (depending on $k$) such that  
\begin{equation}\label{exp.close}
\| Q_\eps - q_\eps\|_{H^k} \le C \eps^{-c/\eps}.
\end{equation}
We will prove 
\begin{lem}
Let $u_\eps$ be the solution of \eqref{nlw} described in Proposition \ref{old.results} below.
Then for every $(y_0, y_1)\in (-T_1,T^1)\times \mathbb S^1$, there is a {\em unique}
$s_*(y_0, y_1)$ 
such that 
\[
 \| v_\eps(y_0, y_1, \cdot) - \tau_{s_*(y_0,y_1)} Q_\eps\|_{L^2(I)}
= \min_{\sigma\in \R}
 \| v_\eps(y_0, y_1, \cdot) - \tau_\sigma Q_\eps\|_{L^2(I)}.
\]

\end{lem}
Note that $s_*(y_0, y_1)$ depends on $\eps$.
For $(y_0, y_1, y_2)\in (-T_1,T^1)\times \mathbb S^1\times I$,
we now define
\[
V_\eps(y_0,y_1,y_2) := Q_\eps(y_2 - s_*(y_0,y_1)).
\]
Thus, $V_\eps$ may be seen as a canonical projection of $v_\eps$
onto the space of functions exhibiting  an almost-canonical\footnote{\ ``almost", 
because of the (exponentially small) difference
between $q_\eps$ and $Q_\eps$.}
interface near $\Gamma$.

Finally, for $(t,x)\in (-T_0,T^0)\times \R^2$ we define
\begin{equation}\label{Uep.def}
U_\eps := \begin{cases}
1 &\mbox{ if }(t,x) \in \mathcal O \setminus \mathcal N\\
V_\eps\circ \Psi^{-1} &\mbox{ if }(t,x)\in \mathcal N\\
-1&\mbox{ otherwise. }
\end{cases}
\end{equation}

We will write $\| \cdot   \|_{H^1_\eps(\Omega)}$ for the norm defined by
\begin{equation}\label{rescaledH1norm}
\| w_\eps \|_{H^1_\eps(\Omega)}^2 
:= \frac 1 \eps \| w_\eps |_{L^2(\Omega)}^2 + \eps \| D w_\eps \|_{L^2(\Omega)}^2
\end{equation}
where $Dw_\eps$ denotes the full gradient in $\Omega$. Thus, for example, if $\Omega$ is an
open subset of $\R_t\times \R^2_x$, then $Dw_\eps = (\partial_t w_\eps, \partial_{x_1} w_\eps, \partial_{x_2}w_\eps)$.

Our main result is

\begin{thm}\label{main.theorem}
Assume that $\Gamma \subset  (-T_*,T^*)\times \R^2$
is a smooth embedded timelike minimal surface
admitting a parametrization of the form \eqref{param},
so that normal coordinates may be defined as in \eqref{normalc}.

For $\eps\in (0,1]$,
let $u_\eps$ be the solution of \eqref{nlw} from Theorem \ref{TheoremA},
described in more detail in Proposition \ref{old.results} below.

Then for every $T_0<T_*$ and $T^0<T^*$, there exists a constant
$C$, independent of $\eps$, such that
\[
\| u_\eps - U_{\eps}\|_{H^1_\eps((-T_0,T^0)\times \R^2)} \le C\eps.
\]
and in addition,
\begin{equation}\label{sstar.H1}
\int_{\mathbb S^1}
s_*^2 + (\partial_{y_0}s_*)^2 + (\partial_{y_1}s_*)^2 dy_1
 \le C \eps^{2} \quad\mbox{ for every }y_0\in (-T_1,T^1).
\end{equation}
\end{thm}

Note since the Minkowskian distance $d$ to $\Gamma$ can be identified with the $y_2$ coordinate, 
we can write $U_\eps$ near $\Gamma$ in the form
$Q_\eps( d - s\circ P)$
where $P$ is the Minkiwskian projection onto $\Gamma$. Since $s$ is small and $Q_\eps$ is very close to $q(\frac \cdot \eps)$, the theorem can be seen as a 
justification (and clarification) of the heuristic principle \eqref{profile1}.

Our arguments could also be used to improve on Theorem \ref{TheoremA}
in dimensions $n\ge 3$. However,
the restriction to $n=2$ dimensions is used in an essential way in Lemma \ref {L.Linf} below, so any such improvements would be much less satisfactory than the
ones we are able to prove for $n=2$.

\subsection{Prior results}\label{priorwork} 

The proof of Theorem \ref{TheoremA} in \cite{galjer, jer0909}
rests on weighted energy estimates for  the solution $v_\eps$ as written
in normal coordinates. These energy estimates,
as mentioned above, provide more information than is 
recorded in Theorem \ref{TheoremA},
and they will provide the starting point for our analysis. 
Before recalling them we introduce some
notation.

We will use the notation
\[
c_{0}:=\int_{-\infty}^{\infty}\left(\frac{\eps}{2}(q_{\eps}')^{2}+\frac{1}{2\eps}(q_{\eps}^{2}-1)^{2}\right)dz \qquad\mbox{ for every $\eps>0$}.
\]
(In fact $c_0 = 4/3$.)
We will write $I := (-\rho, \rho)$.
For a function $v_\eps: (-T_*,T^*)\times \mathbb S^1\times I\to \R$, we will
write
\begin{align*}
\Theta_1( y_0) &:=  \int_{\mathbb S^1}\left[ \int_I (1+y_2^2)\left( \frac \eps 2 (\partial_{y_2}v_\eps)^2 + \frac 1 {2\eps^2}(v_\eps^2-1)^2\right)
dy_2 - c_0\right] dy_1\\
\Theta_2(y_0) &:= 
\int_{\mathbb S^1}\int_I y_2^2 ( v_\eps - \sign(y_2))^2 \, dy_2\,dy_1\\
\Theta_3(y_0) &:= \int_{\mathbb S^1} \int_I \frac \eps2 \left[(\partial_{y_0}v_\eps)^2 + (\partial_{y_1}v_\eps)^2\right] + y_2^2 \left[\frac \eps2 (\partial_{y_2}v_\eps)^2 + \frac 1{2\eps}(1-v_\eps^2)^2\right] dy_2\,dy_1 
\end{align*}
where in every case, $v_\eps$ is understood to be evaluated at the
value of $y_0$ appearing  in the argument of $\Theta_j$.

\begin{proposition}[\cite{jer0909}]\label{old.results}
Assume that $\Gamma \subset  (-T_*,T^*)\times \R^2$
is a smooth embedded timelike minimal surface
admitting a parametrization of the form \eqref{param},
so that normal coordinates may be defined as in \eqref{normalc}.

Then for every $\eps\in (0,1]$, there exists
a solution $u_\eps: \R\times \R^2\to \R$ of the semilinear wave equation \eqref{nlw}
such that \eqref{u_weakly_close_toA} holds, together with
the following estimates:

\smallskip

{\bf 1. Estimates in normal coordinates near $\Gamma$}.
First,  for every $T_0<T^*$ and $T^0<T^*$, there 
exists a constant $C>0$ and a choice of the parameters $\rho,  T_1, T^1$ in the definitions of Domain$(\Psi)$ 
and $\Theta_j, j=1,2,3$
such that $v_\eps := u_\eps \circ \Psi$ satisfies
\begin{equation}\label{Thetas.est}
\Theta_j(y_0)\le C \eps^2 \qquad 
\mbox{ for all $y_0\in (-T_1, T^1)$ and for $j=1,2,3$.}
\end{equation}
$C, T_1,T^1, \rho$ may depend on $\Gamma, T_0, T^0$ but 
are independent of $\eps\in (0,1]$.

\smallskip

{\bf 2. Estimates in  $(t,x)$ coordinates far from $\Gamma$}.
Second, for the same $T_1, T^1, \rho$ and $C$, if we define 
\begin{align*}
 {\mathcal N}' &:= \Psi( (-T_1,T^1)\times \mathbb S^1\times I)\\
\mathcal M &:=  \big( (-T_0,T^0)\times  \R^2\big) \setminus \mathcal N'\\
\mathcal M_t &:=  \{ x\in \R^2 : (t,x)\in \mathcal M\} \ ,
\end{align*}
then $\partial \mathcal M_t$ is uniformly smooth for $ t\in (-T_0,T^0)$, and
\begin{equation}
\int_{\mathcal M} \frac \eps 2 |D u_\eps|^2 + \frac 1 {2\eps}(u_\eps^2-1)^2 dx\,dt \le C \eps^2\ .
\label{uep.far}\end{equation}

\smallskip

{\bf 3. Additional properties}.
Finally, \eqref{u_weakly_close_toA} holds, and there exists some $R>0$ such that 
\begin{equation}\label{comp_supp}
u_\eps = \sign_{\mathcal O} = -1 \qquad\mbox{ for all $(t,x)$ such that }\ t\in (T_*, T^*), \  |x| \ge R.
\end{equation}
\end{proposition}

These are the $n=2$ case of conclusions that are proved\footnote{We note  that there are some cosmetic differences between \cite{jer0909} and Proposition \ref{old.results}.
For example,
in \cite{jer0909} it is assumed for
notational simplicity that $T_* = T^*$ and $T_j = T^j$ for $j=0,1$. The proofs
however make no use of this assumption and remain valid as stated here.}
 in \cite{jer0909}.
More precisely, the relevant initial data are constructed in Lemma 9.
The choice\footnote{In fact one can take any $T_1$ such that $T_0<T_1<T_*$,
and similarly $T^1$, and then arrange that all the required properties hold
by choosing $\rho$ sufficiently small.} 
of $T_1,T^1$ and $\rho$ is described
in Section 2.4. Conclusion \eqref{uep.far} appears in the statement of the main result,
\cite[Theorem 1]{jer0909}. 
It follows from Propositions 10 and 13 that conclusion \eqref{Thetas.est} holds for all $y_0\in [0, \tau]$ for some 
$\tau>0$. In Section 6 (see  \cite[equation (6-17)]{jer0909})
it is shown that \eqref{Thetas.est}  
may be extended to all $y_0\in (-T_1, T^1)$. Estimate
\eqref{u_weakly_close_toA} has already been recalled in Theorem \ref{TheoremA}.
Finally, \eqref{comp_supp} is not explicitly stated in \cite{jer0909}, but it is
a standard consequence of assumptions about the iniitial data
(with $u_\eps(0,x) = -1$ and $\partial_t u_\eps(0,x) = 0$ for $|x|$ outside some
large ball) and finite propagation speed for the wave equation.

For a function $v_\eps:I\to \R$ we will use the notation
\begin{align}
\theta_1(v_\eps) 
&:= 
\ 
 \int_I (1+z ^2)\, \left(\frac \eps 2 v_\eps'(z)^2 + \frac 1{2\eps}(v_\eps^2-1)^2 \right) dz - c_0 
 \label{theta1.def}\\
\theta_2(v_\eps) 
&:= 
 \int_I |z| \ |v_\eps(z) - \sign (z)|^2  dz.
\label{theta2.def}
\end{align}

Our goal is to show that the estimates in Proposition \ref{old.results} in fact
imply the $H^1_\eps$ estimate stated in Theorem \ref{main.theorem}.
In doing this we use from the following fact:

\begin{lem}
There exist positive constants $c_1, c_2$ (depending on $\rho$ only) such that 
\begin{equation}
\mbox{ if }\theta_2(v_\eps) \le c_2, \ \ \mbox{ then } \ \ \ \ \int_I  \frac \eps 2  v_\eps'^2(z)+ \frac 1{2\eps}(v_\eps^2-1)^2 dz - c_0 \ge C e^{-c/\eps}.
\label{coercive1}\end{equation}
Moreover, if in addition $\theta_1(v_\eps) \le c_1$, then
\begin{align}\label{coercive2}
\int_{I}\left(\sqrt{\eps}v_\eps'-\frac{f_{1}(v_\eps)}{\sqrt{\eps}}\right)^{2}dz
& \le  C\left[ \int_I  \frac \eps 2  v_\eps'^2(z)+ \frac 1{2\eps}(v_\eps^2-1)^2 dz -c_0\right] +C e^{-c/\eps} \nonumber
\\
 &\le C \theta_1(v_\eps) + C e^{-c/\eps},
\end{align}
where $f_1(v_\eps) = 1-v_\eps^2$.
\label{L.gsj1}\end{lem}
This follows directly from Lemma 5.3 in  \cite{galjer}.

\section{A canonical decomposition}

The main result of this section is the following:

\begin{proposition}\label{prop1} There exists $\delta>0$ such that
if 
 $\ds{\inf_{|\sigma|\le \rho/6 }
 \| v_\eps - \tau_\sigma Q_{\eps}\|_{L^2(\R)}}\leq\delta\sqrt{\eps}$
 then
there is a {\em unique} $s_*\in \R$ 
such that 
\[
\| v_\eps - \tau_{s_*} Q_{\eps}\|_{L^2(I)}
= \min_{\sigma\in \R}
 \| v_\eps - \tau_\sigma Q_{\eps}\|_{L^2(I)}.
\]
Moreover, 
\begin{equation}
\int_I (v_\eps - \tau_{s_*} Q_\eps)\cdot \tau_{s_*} Q_\eps' \ = 0.
\label{optttr}\end{equation}
\end{proposition}

Results in this spirit are in some sense standard, 
but for the convenience of the reader we give a quick proof.

\begin{proof}
Let
\[
\varphi (\sigma) :=  \| v_\eps - \tau_\sigma Q_\eps\|_{L^2(I)},
\qquad\qquad
\eta (\sigma) := \frac 12 \| v_\eps - \tau_\sigma Q_\eps\|_{L^2(I)}^2.
\]
The continuity of translation in $L^p$ spaces implies that
$\varphi$ and $\eta $ are continuous. 
In addition, for any $\sigma\in \R$, the triangle
inequality implies that
\begin{equation}
\| \tau_\sigma Q_\eps - \tau_{s}Q_\eps\|_{L^2(I)}
- \varphi(s) 
\ \le\  \varphi(\sigma)\le
\| \tau_\sigma Q_\eps - \tau_{s}Q_\eps\|_{L^2(I)}
+ \varphi(s) 
\label{varphi1}\end{equation}
Define $f(t) := \| \tau_t q - q\|_{L^2(\R)}^2$. 
Recalling that $Q_\eps(z) = \sign(z)$ for $|z|\ge \frac 23\rho$,
and using \eqref{exp.close} and a change of variables,
\[
\| \tau_\sigma Q_\eps - \tau_{s}Q_\eps\|_{L^2(I)} = 
\| \tau_{\sigma -s} Q_\eps - Q_\eps\|_{L^2(\R)} = \left(\eps f( \frac {|\sigma - s|}\eps)\right)^{1/2} + O(e^{-c/\eps}) 
\]
if $|s|,|\sigma| \le \frac \rho 3$.
Also, it is straightforward to check that $f$ is smooth, with $f(0)=f'(0)=0$ and $f''(0):= 2a >0$
(in fact $a = \int_\R q'^2 = 4/3$.)
It follows that there exists a positive number $\delta_1$ such that 
\begin{equation}
 \frac a{2 \sqrt \eps} |s-\sigma| \ \le \ 
\| \tau_\sigma Q_\eps - \tau_{s}Q_\eps\|_{L^2(I)} + O(e^{-c/\eps})  \  \le \ 
 \frac {2a} {\sqrt \eps} |s-\sigma| 
\label{L2dQ1}\end{equation}
if $|s|,|\sigma| \le \frac \rho 3$ and $|s-\sigma|<\delta_1 \eps$.
Also, since $\sigma \mapsto \| \tau_\sigma Q_\eps - \tau_{s}Q_\eps\|_{L^2(I)} $ is nondecreasing in $|s-\sigma|$, 
\begin{equation}
\| \tau_\sigma Q_\eps - \tau_{s}Q_\eps\|_{L^2(I)} \ge \frac {\delta_1 a}3 \sqrt \eps
\qquad
\mbox{ if $|s|\le \frac \rho 3$ and $|s-\sigma| \ge \delta_1 \eps$ .}
\label{L2dQ2}\end{equation}

By hypothesis, there exists some $s_1$ such that 
$\varphi(s_1)<\delta \sqrt \eps$ and $|s_1 |\le \rho/6$.
Then \eqref{varphi1} and \eqref{L2dQ2} imply that for any $\sigma\in \R$
\begin{equation}\label{qqqq}
\varphi(\sigma) \ge  \frac {\delta_1 a}3 \sqrt \eps - \delta \sqrt \eps \ge \delta\sqrt \eps > \inf \varphi\qquad\mbox{ if }
|\sigma - s_1| \ge \delta_1\eps,
\end{equation}
as long as $\delta < \frac {\delta_1 a}6$. It follows that $\min \varphi$ is attained at some 
$s_*$, and that $|s_* - s_1 | < \delta_1 \sqrt \eps$.

Also, one can easily check using the dominated convergence theorem that 
if $|\sigma| < \frac 13 \rho$ (and thus $Q_\eps(z) = \sign(z)$ in a neighborhood
of the endpoints of $I$, see \eqref{Qeps.def}) then
\[
\eta'(\sigma)
=
\int_I (v_\eps - \tau_{\sigma} Q_\eps)\cdot \tau_{\sigma} Q_\eps' .
\]
Thus equation \eqref{optttr}
follows directly from the optimality of $s_*$.

It remains to prove the uniqueness of the minimizer $s_*$.  Let $\sigma$ be any minimizer 
of $\varphi$. Arguing as in \eqref{qqqq}, we find that $|s_* - \sigma| < \delta_1\eps$.
Repeating the same argument, but now using \eqref{L2dQ1} in place of \eqref{L2dQ2},
we find that $|s_*-\sigma| < 4\eps \delta /a $


To complete the proof, it therefore suffices to show that
if $\delta$ is small enough, then $\eta$ is 
strictly convex
in the interval $(s_* -  4 \eps \delta/a, s_*+4\eps\delta/a)$, and
hence in this interval can only attain its minimum at a single point,
necessarily $s_*$. 

To check convexity, we use the dominated convergence theorem as above to
compute
\[
\eta''(\sigma) = \int_I (\tau_\sigma Q_\eps')^2 - \int_I (v_\eps - \tau_\sigma Q_\eps) \tau_\sigma Q_\eps''.
\]
Using \eqref{exp.close}, we check that if $\eps$ is small enough, then for $|\sigma|< \frac 13 \rho$, 
\[
\| \tau_\sigma Q_\eps'\|_{L^2(I)}^2 \ge \frac {c_0}{2\eps}, \qquad \qquad
\| \tau_\sigma Q_\eps''\|_{L^2(I)} \le \frac C {\eps^{3/2}}.
\]
In addition, if $|\sigma - s_*|\le 4\eps\delta/a < \delta_1\eps$, then we know from \eqref{varphi1}  and \eqref{L2dQ1} that $\varphi(\sigma) \le 9 \delta \sqrt{\eps}$, and thus
\[
\eta''(\sigma) \ge \frac {c_0}\eps - \| v_\eps - \tau_\sigma Q_\eps\|_{L^2(I)} \, \|  \tau_\sigma Q_\eps''\|_{L^2(I)}\ \ge \ \frac {c_0}\eps  - C \frac \delta\eps .
\]
The right-hand side can be made positive by decreasing $\delta$, if necessary.
This proves convexity of $\eta$ when $|\cdot - s_*|\le 4\delta\eps/a$ and hence completes the uniqueness proof.
\end{proof}

\section{Coercivity of $\theta_1$}

The main result of this section shows that under suitable hypotheses,
$\theta_1(v_\eps)$ controls the $H^1_\eps$ norm of $v_\eps$ and the size
of the optimal translation $s_*$.

\begin{proposition}\label{proposition1}There exist positive constants 
$c_1, c_2, c_3$ such that $0<c_3<1$, and for every $\theta\in H^1(I)$, if either
\begin{equation}\label{theta1c1}
\theta_1(v_\eps)\le c_1, \qquad \theta_2(v_\eps)\le c_2
\end{equation}
or 
\begin{equation}\label{L2c3}
\inf_{|s| \le  c_3 \rho} \| v_\eps - \tau_s Q_\eps\|_{L^2(I)} < c_3 \sqrt \eps\ ,
\end{equation}
then for all sufficiently small $\eps$, then there is a unique minimizer $s_*$ 
of $\varphi(s) := \|v_\eps - \tau_s Q_\eps\|_{L^2(I)}$, and
\begin{align}
s_*^2 
&\lesssim \eps^2 + \theta_1(v_\eps)  \label{sstar}\\
\| v_\eps - \tau_{s_*} Q_\eps\|_{H^1_\eps(I)}^2
&\lesssim  \theta_1(v_\eps)
+ e^{-c/\eps}.
\label{gp.est}
\end{align}
\end{proposition}

Estimates in the spirit of \eqref{gp.est} are known, but we do not know a source where
they are proved under the hypotheses that we
impose here, so we give a self-contained proof.

\vskip0.25cm

  The rest of this section is devoted to the

\begin{proof}[Proof of Proposition \ref{proposition1}]
We will first prove the proposition under the assumption
\eqref{theta1c1}, for constants $c_1,c_2$ to be fixed below. 
At the end of the proof, we will consider assumption \eqref{L2c3}.

First, we define $h_\eps : (-\rho,\rho)\to \R$ by
\begin{equation}
 v_\eps' - \frac 1 \eps f_1(v_\eps) =: h_\eps.
\label{heps.def}\end{equation}
Then it follows  from \eqref{coercive2}  that 
\begin{equation}\label{vandf_1}
\int_{-\rho}^{\rho} \eps h_\eps^2 \, dz \ \lesssim 
\theta_1(v_\eps) + C e^{-c/\eps} \lesssim c_1
\end{equation}
for $\eps$ small.

  It is convenient to extend $h$ to the entire real line, by setting $h_\eps = 0$
outside of $(-\rho, \rho)$, and to extend $v_\eps$ by requiring that the ODE \eqref{heps.def}
holds on the entire real line.  This will allow us to translate $v_\eps$ without worrying about
redefining its domain. We continue to use the notation $v_\eps$ and $h_\eps$
for the extended functions.

  It is straightforward to check that if $c_2$ is small enough (depending on $\rho)$, then since 
$v_\eps\in H^1(I)\subset C(I)$, the hypothesis
$\theta_2(v_\eps)\le c_2$ 
implies that 
\begin{equation}
v_\eps(s_0)=0\qquad \qquad \mbox{ for some $|s_0|\le \rho/2$}.
\label{p1.1}\end{equation}
We will prove that 
\begin{equation}\label{weps.est}
\| w_\eps\|_{H^1_\eps(\R)}^2 \le C{\eps}\|h_\eps\|_{L^2(\R)}^2 \,, \qquad
\mbox{ for }w_\eps := v_\eps - \tau_{s}q_\eps.
\end{equation}

We will see that \eqref{gp.est} is easily deduced from this.

Note that 
$w_\eps= v_\eps - \tau_{s_0} q_\eps$ vanishes at $s_0$ and  recall that 
 $\tau_{s_0} q_{\eps}$
 satisfies
\[(\tau_{s_0}q_\eps)' - \frac{1}{\eps}f_{1}(\tau_{s_0}q_\eps) = 0.\]
By subtracting the latter equation from  \eqref{heps.def},
which is satisfied by $v_\eps$, we get
$$\begin{array}{ll}
w_\eps'&=\ds{\frac{1}{\eps}f_{1}(v_\eps)+h_\eps-\frac{1}{\eps}
f_{1}(\tau_{s_0}q_\eps)}=\ds{\frac{1}{\eps}((\tau_{s_0}q_\eps)^{2}-v_\eps^{2})+h_\eps}\vspace{5 pt}
\\
&=\ds{\frac{1}{\eps}(-w_\eps)(\tau_{s_0}q_\eps+v_\eps)+h_\eps}.
\end{array}$$
Thus, $w_\eps$ satisfies the ordinary differential equation 
\begin{equation}\label{NLODEa}
\left\{\begin{array}{rl}
w_\eps'=& \ds{\frac{1}{\eps}}\ds{(-w_\eps^{2}-2 w_\eps\, \tau_s q_\eps)+h_\eps}\text{ on }\R , \vspace{4 pt}\\ 
w_\eps(s_0)=&0.
\end{array}\right.
\end{equation}
We write the above problem in a more convenient form via an appropriate rescaling of the functions. Namely,  
\[w(z) := w_\eps(\eps(z-s_0)) \text{ and }h(z) := \eps h_\eps(\eps(z-s_0)).\] Then we have
\begin{equation}\label{NLODE}
\left\{\begin{array}{rl}
w'=&  \ds{-( 2q+ w)w+h}\text{ on }\R  \\ 
w(0)=&0.
\end{array}\right.
\end{equation}
Moreover, it follows from \eqref{vandf_1} that, if $\eps$ is small, then
\begin{equation}
\| h\|_{L^2(\R)}^2 \ = \  \eps \| h_\eps\|_{L^2(\R)}^2  \ \lesssim \  \theta_1(v) + Ce^{-c/\eps}  \ \lesssim \  c_1.
\label{heps.small}\end{equation}
Since $\| w_\eps\|_{H^1_\eps(\R)} = \| w\|_{H^1(\R)}$, it now suffices to estimate the
latter quantity.

To do this, we will show via the contraction mapping principle that if $c_1$ is small then
\eqref{NLODE} admits a unique solution which satisfies
\begin{equation}\label{sts}
\| w\|_{H^1(\R)}^2 \le C \|h \|_{L^2(\R)}^2,
\end{equation}
which is the same as \eqref{weps.est}, after rescaling.

  We set 
\begin{equation}\label{Ballalpha}
\forall \alpha>0,
~B_{\alpha}:=\left\{w\in H^{1}(\R) ,\text{ such that }\|w\|_{H^{1}(\R)}\leq\alpha\right\}.
\end{equation}
In order to use the contraction mapping principle on $B_{\alpha}$, we  define the following operator $\mathcal{S}$:

\begin{defin}\label{definoperator} Given $w_{0}\in B_{\alpha},$ we define  $\mathcal{S}(w_{0}):=w_{1}$ to be the solution of  
\begin{equation}\label{linearizedODE}
\left\{\begin{array}{rl}
w_{1}'=&-(2q+w_{0})w_{1}+h \qquad \text{ on }\R,\vspace{5 pt}\\
w_{1}(0)=&0.
\end{array}\right.
\end{equation}
\end{defin}

  We prove the following result:

\begin{lem}\label{lemma1a}
Let $\mathcal{S}$ be the operator defined in (\ref{linearizedODE}) above.
There exists a constant $C$ such that  
\[
\mbox{ if }\|w_0\|_{H^1} \le \sqrt{2}, \mbox{ then }  \ \ \ \| \mathcal{S}w_0\|_{H^1} \le C \| h \|_{L^2}.
\]
\end{lem}

\begin{proof}[Proof of Lemma \ref{lemma1a}]
  For each $s$, we set  $$\Phi(s):=\int_{0}^{s}(2q(t)+w_{0}(t))~dt.$$
Then we have the explicit fromula
$$w_{1}(s)=e^{-\Phi(s)}\left(\int_{0}^{s}e^{\Phi(t)}h(t)~dt\right)$$
which leads us to write 
\begin{equation}\label{solutionoperator}
\mathcal{S}(w_{0})(s)=w_{1}(s)=\ds{\int_{0}^{s}\exp\left(-\int_{t}^{s}(2q(\tau)+w_{0}(\tau))~d\tau\right)h(t)~dt}.
\end{equation}
To prove our claim about the map $\mathcal{S}$, we use first the $1$-dimensional Sobolev embedding (with sharp constant $\frac 12$) to note that
\begin{equation}
\|w_0\|^{2}_{L^\infty} \le \frac 12 \| w_0\|^{2}_{H^1} \le 1\
\label{w0Linfty}\end{equation}
if $ \|w_0\|_{H^1}\le \sqrt{2}$, which we henceforth assume to hold.
Thus
$2q+w_{0}\geq2q-1,$ or 
$$-(2q+w_{0})\leq-2q+1.$$ Thus for $s\ge 0$, 
\begin{equation}
|\mathcal{S}(w_{0})(s)|\leq\ds{\int_{0}^{s}\exp\left(\int_{t}^{s}(-2q(\tau)+ 1)~d\tau\right)|h(t)|~dt}.
\end{equation}

Using the explicit form of $q$, we can integrate to find that for any $s \ge 0 $, 
\begin{align}
|\mathcal{S}w_{0}(s)|&\leq\ds{\int_{0}^{s}\exp\left(\ds{2\ln\left(\frac{\cosh t}{\cosh s}\right)+(s-t)}\right)|h(t)|~dt}\vspace{5 pt}\nonumber\\
&=\ds{\int_{0}^{s}\left(\frac{\cosh t}{\cosh s}\right)^{2} e^{s-t} |h(t)|~dt}\vspace{6 pt}
\label{Swa}\end{align}
Since $\frac 12 e^a \le \cosh a \le e^a$ for $a>0$, it follows that
\begin{align*}
\mathcal{S}w_{0}(s) 
&\le  4\int_0^s e^{2(t-s)} e^{s-t} |h(t)| \ dt  \\
&=
4E * H (s)
\end{align*}
for $s>0$, where
\[
E(s) := \begin{cases} 
e^{ -s} &\mbox{ if }s> 0\\
0&\mbox{ if }s\le 0
\end{cases},
\qquad
H(s) := \begin{cases} 
|h(s)| &\mbox{ if }s> 0\\
0&\mbox{ if }s\le 0.
\end{cases}
\]
Then it follows from Young's inequality that for any $p\ge 2$,
\[
\frac 14 \| {\bf 1}_{s>0 } \  \mathcal{S}w_0\|_{L^p} \le  \| E\|_{L^q} \ \| H \|_{L^2}  =  \| h {\bf 1}_{s>0 }\|_{L^2},\qquad \frac 1q = \frac 1p + \frac 12 .
\]

  Since the same arguments (with some changes of sign) apply to $ {\bf 1}_{s<0 } \  \mathcal{S}w_0$, 
it follows that for any $p\ge 2$,
\begin{equation}
\|  \mathcal{S}w\|_{L^p}   \le   4 \| h\|_{L^2} \qquad \mbox{ for all }w\in B_\alpha \label{SwL2}\end{equation}
as long as $\alpha$ is sufficiently small ($\alpha\leq \sqrt{2}$). Lastly, we note that the definition  of $\mathcal S$ in \eqref{linearizedODE}, together with \eqref{SwL2}
and the Sobolev embedding, leads to
\begin{align}
\nonumber \|( \mathcal{S}w)' \|_{L^2}
&\le (2 + \|w\|_\infty) \| \mathcal{S} w\|_{L^2} + \|h\|_{L^2}\\
&\le (2 + C \alpha)C \|h\|_{L^2} + \|h\|_{L^2}
\le C \|h\|_{L^2}. \label{Sw'}
\end{align}
The estimates \eqref{SwL2} and \eqref{Sw'} finish the proof of Lemma \ref{lemma1a}.
\end{proof}
\begin{rem}
The Sobolev embedding with sharp constant $1/2$, which we used in  the proof of Lemma \ref{lemma1a} above, allows us to see how small the radius $\alpha$ of the  $H^1$ ball $B_\alpha$ could be--independently of any parameters ($\alpha\leq \sqrt{2}$). 

\end{rem}
\begin{corl}\label{corollary1} 
There exists a constant $C>0$ such that for any $\alpha\in (0,\sqrt2]$, 
\[
\mbox{ if }\|h \|_{L^2(\R)} 
< C^{-1}\alpha \quad  \text{ and }\quad 0<\alpha\leq \sqrt 2,
\]
then $\mathcal S(B_\alpha)\subseteq B_\alpha$. That is, the $H^1$-ball $B_\alpha$ is stable under $\mathcal{S}.$
\end{corl}
\begin{proof} Let $w\in B_\alpha$ with $\alpha\leq \sqrt 2.$  Then 
Lemma  \ref{lemma1a} implies that there exists a constant $C$ such that 
$\|\mathcal S(w)\|_{H^1(\R)}\leq C\|h \|_{L^2(\R)} $.
Thus $\mathcal S(w)\in B_\alpha$ if $\|h \|_{L^2(\R)}< C^{-1}\alpha$ (with the same
constant $C$.
\end{proof}
We next  prove that if $\|h\|_{L^2}$ is sufficiently small, then there
exists some $\alpha>0$ such that $\mathcal S$ is a contraction mapping on $B_\alpha$.

\begin{lem}\label{lemma1}
Let $\mathcal{S}$ be the operator defined in (\ref{linearizedODE}) above.
Then there exist constants $C, \alpha_0>0$ such that 
\[
\mbox{ if }\|h \|_{L^2}  \le \alpha_0 
\]
then the map $\mathcal S$ is a contraction mapping of $B_\alpha$ to itself, for
$\alpha = C \| h\|_{L^2}$.

  Hence, if $\| h\|_{L^2} \le\alpha_0$, then the unique solution $w$ of the initial
value problem \eqref{NLODE} satisfies \eqref{sts}.
\end{lem}

\begin{proof} 

Assume that $\| h\|_{L^2} < \alpha_0$ (to be adjusted below) and set $\alpha = C\| h\|_{L^2}$, for the same $C$ as in Corollary \ref{corollary1}. We require $\alpha_0$ to be small enough that $\alpha \le\sqrt 2$; then
Corollary \ref {corollary1} applies, and it guarantees that
$\mathcal S(B_\alpha) \subset B_\alpha$.

Let $w , \hat{w}\in B_\alpha$   and set $\mathcal{S}w=w_1$, $\mathcal{S}\hat{w}=\hat{w}_1$ and $v:=w_1-\hat{w}_1.$  The main point in the proof of this lemma is to get estimates on $\|v'\|_{L^2(\R)}$ in terms of $\|w-\hat w\|_{H^1(\R)}.$
We write  \eqref{linearizedODE} for $w$ and $\hat w$, and get
\[\begin{array}{ll}
(w_{1}-\hat{w}_1)'&=-(2q+w)w_1+(2q+\hat w)\hat w_1\vspace{5 pt}\\
&=-(2q+\hat w)(w_1 - \hat w_1)+w_1(\hat w-w).
\end{array}
\] 
That is, \begin{equation}\label{towards.cont1}
v'=-(2q+\hat w)v+(\hat{w}-w){w}_{1}.
\end{equation}
Since $w_1(0) = \hat w_1(0)= 0$, it follows that 
$v$ solves \eqref{linearizedODE} with $w_0$ replaced
by $\hat w$ and $h$ replaced by $w_1(\hat w-w)$.
By assumption, $\| \hat w\|_{H^1} <\alpha \le \sqrt 2$, so 
we may use Lemma \ref{lemma1a} to conclude that
 \[
 \begin{array}{rll}
 \|v\|_{H^1}&\leq C\| ( \hat w - w) w_1\|_{L^2} &\\
& \leq C \|w-\hat{w}\|_{L^2}\|\hat{w_1}\|_{H^1}\
& \mbox{ (using the embedding $H^1\hookrightarrow L^\infty$),}\\
&\leq C\alpha_0  \|w-\hat{w}\|_{H^1} 
\qquad  &\mbox{ (since $\hat w_1 = \mathcal S \hat w \in B_\alpha$ and $\alpha = C \|h\|_{L^2}\le C \alpha_0$).}\end{array}
\]

For this choice of $\alpha_0$, if $\|h\|_{L^2}<\alpha_0$, there is a unique
fixed point $w$ of $\mathcal S$ in $B_\alpha$, and this  clearly solves \eqref{NLODE}
and satisfies the estimate we are seeking, {\em i.e. } $\|w\|_{H^1}\le C \|h\|_{L^2}$.
On the other hand, the initial value problem \eqref{NLODE} has a
unique solution as long as that solution remains bounded. It follows that this
solution agrees with the fixed point $w$ of $\mathcal S$. Consequently,
the solution $w$ of \eqref{NLODE} satisfies \eqref{sts}.

Thus $\mathcal S: B_\alpha\to B_\alpha$ is a contraction mapping if  (in addition to the smallness condition imposed above) $\alpha_0$ is small enough that $C\alpha_0<1$.
\end{proof} 

We break the remainder of the proof of Proposition \ref{proposition1} into
several small pieces.

\begin{proof}[Proof that \eqref{theta1c1} implies \eqref{gp.est}]
Assume that \eqref{theta1c1} holds, for 
$c_1, c_2>0$ no greater than the constants of the same name in Lemma \ref{L.gsj1}, and such that $c_2$ implies \eqref{p1.1}. In addition, in view
of \eqref{heps.small}, we can fix $\eps_0>0$ such that,
after taking $c_1$ smaller if necessary, we have
$\| h\|_{L^2} \le \alpha_0$ whenever $\theta_1(v_\eps)\le c_1$ and $0<\eps<\eps_0$.
It then follows from Lemma \ref{lemma1} that \eqref{sts} holds, and hence 
\eqref{weps.est}.

Now let $s_*$ minimize $\| v_\eps - \tau_sQ_\eps\|_{L^2(I)}$.
(It is clear that the minimum  is attained,
since $\tau_sQ_\eps(z) = -\sign(z)$ for all $z\in I$ whenever $|s| \ge 2\rho$.)
Let $W_\eps := v_\eps - \tau_{s_*} q_\eps$. 
Then from the optimality of $s_*$,
because $Q_\eps$ and $q_\eps$ are exponentially close,
and using  \eqref{weps.est}, we have
\begin{align}
\frac 1 \eps\|W_\eps\|_{L^2(I)}^2
\le
\frac 1 \eps\| v_\eps - \tau_{s_0}q_\eps\|_{L^2(I)}^2 + C e^{-c/\eps}
&= \frac 1 \eps \| w_\eps\|_{L^2}^2 + C e^{-c/\eps}\nonumber \\
&\lesssim \theta_1(v_\eps) + C e^{-c/\eps}.
\label{Weps1}\end{align}
Also, exactly as in the argument leading to \eqref{NLODEa}, 
$W_\eps$ satisfies
\begin{equation}\label{Weps2}
W_\eps' = -\frac 1 \eps(2\tau_{s_*}q_\eps + W_\eps) W_\eps + h_\eps
\end{equation}
for the same $h_\eps$ defined in \eqref{heps.def} (but without the initial condition
in \eqref{NLODEa}.) Note that
\[
\|W_\eps\|_{L^\infty(I)} = \| w_\eps + \tau_{s_0}q_\eps - \tau_{s_*}q_\eps \|_{L^\infty(I)} \le \|w_\eps\|_{L^\infty(I)}+2 \le \alpha+2.
\]
We thus see from \eqref{Weps1}, \eqref{Weps2}, and \eqref{vandf_1} that 
\[
\sqrt \eps \| W_\eps'\|_{L^2(I)} \le \frac 1 {\sqrt{\eps}} (\alpha+4) \| W_\eps\|_{L^2(I)} + 
\sqrt \eps \| h_\eps\|_{L^2(I)}
\le C\big( \theta_1(v_\eps) + C e^{-c/\eps})^{1/2}.
\]
By combining this with \eqref{Weps1} and recalling \eqref{exp.close} that $q_\eps$ and $Q_\eps$ are exponentially close, we conclude that \eqref{gp.est} holds.
\end{proof}

\begin{proof}[Proof that \eqref{L2c3} implies \eqref{gp.est}]
Now assume \eqref{L2c3} instead of \eqref{theta1c1}.
We  fix $c_3$ such that \eqref{L2c3} implies that
$\theta_2(v_\eps) \le c_2$ for all sufficiently small $\eps>0$.
It is easy to check that this can be done.

With this choice, we may assume that $\theta_1(v_\eps)\ge c_1$, as
otherwise conclusion \eqref{gp.est}
is already known to hold, by our arguments above.

We define $s_*$ as above. 
It follows directly from \eqref{L2c3}  that
\[
\frac 1 \eps\| v_\eps - \tau_{s_*} Q_\eps\|_{L^2(I)}^2 \le c_3^2.
\]
Then
\begin{align*}
\| v_\eps - \tau_{s_*} Q_\eps\|_{H^1_\eps(I)}^2
&=
\frac 1 \eps \| v_\eps - \tau_{s_*} Q_\eps\|_{L^2(I)}^2+ \eps \| (v_\eps - \tau_{s_*} Q_\eps)'\|_{L^2(I)}^2
\\
&\le 
c_3^2 + 2\eps \| v_\eps ' \|_{L^2(I)}^2 + 2\eps \| Q_\eps'\|_{L^2(\R)}^2.
\end{align*}
Since $\theta_1(v_\eps)\ge c_1$,  and since $c_1, c_3$ are fixed, it is clear that
\[
c_3 = ( \frac{c_3}{c_1}) c_1 \le C \theta_1(v_\eps) , \qquad  \eps \| Q_\eps'\|_{L^2(\R)}^2 = c_0 + Ce^{-c/\eps} \le C c_1\le  C \theta_1(v_\eps)
\]
for $C$ independent of $\eps$, as long as $\eps$ is small. Moreover,
\begin{align*}
\int_{\rho}^\rho \eps v_\eps'^2  
&\le
2 \left(\int_{-\rho}^\rho \frac \eps 2 v_\eps'^2 + \frac 1 {2\eps}(v_\eps^2-1)^2 dz - c_0 \right) + 2c_0
\\
&\le
2 \theta_1(v_\eps) + 2c_0\\
&\le
C\theta_1(v_\eps) .
\end{align*}
again using the fact that $c_0 \le C c_1 \le C \theta_1(v_\eps)$. We obtain \eqref{gp.est} by combining these inequalities.
\end{proof}

{\em Proof of \eqref{sstar}}.
We now prove that $s_*^2 \lesssim \eps^2 +\theta_1(v)$,
if  $s_*$ minimizes $\| v_\eps - \tau_s Q_\eps\|_{L^2(I)}$.
For this, it is convenient to write $Q_\eps^{s_*} := \tau_{s_*}Q_\eps$ and $W_\eps := v_\eps - Q_\eps^{s_*}$.
If $\theta_j(v)\le c_j$ for $j=1,2$, then it follows from \eqref{coercive1} that
\begin{align*}
\theta_1(v_\eps) &\ge
 \int_{-\rho}^\rho
z ^2 \left(\frac \eps 2 v_\eps'\,^2 + \frac 1{2\eps}(v_\eps^2-1)^2 \right) dz\\
&=
 \int_{-\rho}^\rho
z ^2 \left(\frac \eps 2 (Q_\eps^{s_*} + W_\eps)'\,^2 + \frac 1{2\eps}((Q_\eps^{s_*} + W_\eps)^2-1)^2 \right) dz.
\end{align*}
Discarding a positive term and using the inequality
$(a+b)^2 \ge \frac 12 a^2 - b^2$, we conclude that
\[
\theta_1(v) \ge  \int_{-\rho}^\rho \frac \eps 4  z^2 (Q_\eps^{s_*})'^2 \ dz - \int_{-\rho}^\rho\frac \eps 2 W_\eps'^2 \, dz,
\]
or upon rearranging and using \eqref{gp.est},
\[
 \int_{-\rho}^\rho \frac \eps 4 z^2 (Q_\eps^{s_*})'^2 \ dz \lesssim \theta_1(v).
\]
Since
\[
\int_{-\rho}^\rho \frac \eps 4 z^2 (Q_\eps^{s_*})'^2 \ dz  
\ \ge \ (\min_{|z-s|\le\eps} z^2) 
\int_{s_*-\eps}^{s+\eps} \frac \eps 4 (Q_\eps^{s_*})'^2 \ dz  
\gtrsim  \ 
\min_{|z-s_*|\le\eps} z^2 \  \ge \frac 12 s_*^2 - \eps^2
\]
we conclude that $s_*^2 \lesssim \eps^2+ \theta_1(v)$.

\end{proof}

\begin{proof}[Proof of the uniqueness of $s_*$]
Finally, by taking $c_1, c_3$ smaller if necessary, we can arrange (in view of \eqref{gp.est} and \eqref{sstar})
that either \eqref{theta1c1} or  \eqref{L2c3} implies the
hypothesis of Proposition \ref{prop1}, which is that  ${\inf_{|\sigma|\le \rho/6 }
 \| v_\eps - \tau_\sigma Q_{\eps}\|_{L^2(\R)}}\leq\delta\sqrt{\eps}$. 
The uniqueness of $s_*$ then follows.

  This completes the proof of Proposition \ref{proposition1}.\end{proof}

\section{Proof of Theorem  \ref{main.theorem}}

In this section we use Proposition \ref{proposition1} and the results from \cite{galjer, jer0909} recalled in Section \ref{priorwork} to complete the proof of our main result.

We assume that $\Gamma, T_0,T^0$ are given, and 
that $u_\eps$ is the solution of \eqref{nlw} described in Proposition \ref{old.results}.
We fix 
$\rho, T_1, T^1$ as in Proposition \ref{old.results}, and we recall that $v_\eps := u_\eps \circ \Psi^{-1}:(-T_1,T_1)\times \mathbb S^1\times I\to \R$.

\begin{lem}\label{L.Linf}
If $\eps$ is sufficiently small, then for {\em every} $(y_0, y_1)\in (-T_1,T_1)\times \mathbb S^1$,
\begin{equation}\label{L2again}
\inf_{|s| \le c_3\rho } \| v_\eps(y_0, y_1, \cdot)  - \tau_s Q_\eps\|_{L^2(I)} \lesssim \eps^{3/4}.
\end{equation}
As a result, $\inf_{|s| \le c_3\rho } \| v_\eps(y_0, y_1, \cdot)  - \tau_s Q_\eps\|_{L^2(I)} \le c_3\sqrt{\eps}$, where $c_3$ is the constant in the hypothesis \eqref{L2c3} of Proposition \ref{proposition1}.
\end{lem}

This lemma is the only point where we need the assumption that $n=2$;
our argument relies on a $1d$ Sobolev embedding $C^{1/2} \hookrightarrow H^1$
in the tangential variable $y_1$.

\begin{proof}
For $y_0\in (-T_1,T^1)$, we define
\[
\mathcal G(y_0)  :=  \{ y_1\in \mathbb S^1  : \theta_j(v_\eps(y_0, y_1, \cdot)) \le \eps^{3/2} \mbox{ for } j = 1,2 \},
\]
If $y_1\in \mathcal G(y_0)$, then $v_\eps(y_0, y_1, \cdot)$ satisfies the
hypothesis \eqref{theta1c1} of Proposition \ref{proposition1},
and as a result, $s_*(y_0,y_1) \lesssim \eps^{3/4}$ and 
\begin{equation}\label{gs7}
\| v_\eps(y_0, y_1, \cdot) - \tau_{s_*(y_0,y_1)}Q_\eps\|_{H^1_\eps(I)} \lesssim \theta_1(
v_\eps(y_0, y_1, \cdot)) + C e^{-c/\eps} \lesssim  \eps^{3/2}.
\end{equation}
In particular, \eqref{L2again} holds.
So we only need to show that \eqref{L2again} still holds
for $(y_0,y_1)$ if $y_1\in 
\mathcal {B}(y_0) := \mathbb S^1 \setminus \mathcal G(y_0)$.

Toward this end, first note that 
for $j=1,2$, 
\begin{equation}\label{Th.th}
\Theta_j(y_0) = \int_{\mathbb{S}^1}\theta_j(v_\eps(y_0,y_1, \cdot)) dy_1.
\end{equation}
This is a direct consequence of the definitions (see Section \ref{priorwork}).
We also know from Proposition \ref{old.results} that $\Theta_j(y_0)\le C \eps^2$
for $j=1,2$ and for all $y_0\in  (-T_1,T^1)$.
It therefore follows via Chebyshev's inequality that 
\begin{equation}\label{goodset1} 
|\mathcal {B}(y_0)|\leq C \eps^{1/2} \qquad\mbox{ for all }y_0\in (-T_1,T^1).
\end{equation}

We now fix $(y_0, y_1^b)$ such that $y_1^b\in \mathcal B(y_0)$. In view of \eqref{goodset1}, we can find some $y_1^g\in \mathcal G(y_0)$ such that
$|y_1^b - y_1^g| \le C \eps^{1/2}$.
Let us write $s_*^g := s_*(y_0, y_1^g)$. 
Then using \eqref{gs7}
and the triangle inequalty,
\begin{equation}\label{bs8}
\| v_\eps(y_0, y_1^b ,\cdot) - \tau_{s_*^g}Q_\eps\|_{L^2(I)}
\le
\| v_\eps(y_0, y_1^b ,\cdot) -  v_\eps(y_0, y_1^g ,\cdot) \|_{L^2(I)} + C \eps^{5/2}.
\end{equation}
Next, we use the Fundamental Theorem of Calculus and the Cauchy-Schwarz inequality to compute
\begin{align*}
|v(y_{0},y_{1}^g,y_{2})-v(y_{0},y_{1}^b,y_{2})|^2
&=\ds{\left| \int_{y_{1}^g}^{y_1^b}\frac{\partial v}{\partial y_{1}}(y_{0},r,y_{2})~dr \right|^2}
\\
&
\le
\left|y_{1}^g-y_{1}^b\right|\ds{\int_{y_{1}^g}^{y_1^b}\left|\frac{\partial v}{\partial y_{1}}(y_{0},r,y_{2})\right|^2\,dr}.
\end{align*}
We integrate over $y_{2}$ and use \eqref{Thetas.est} to find that 
\[
\|v(y_{0},y_{1}^g,\cdot)-v(y_{0},y_{1}^b,\cdot)\|^2_{L^2}
\lesssim \eps^{1/2} \frac{\Theta_3(y_{0})}{\eps} \lesssim \eps^{3/2} \ .
\]
Since $|s_*^g|\lesssim \eps^{3/4}$, this fact and \eqref{bs8} together imply that
$(y_0,y_1)$ satisfy \eqref{L2again}.
\end{proof}

We will also need the following Sobolev-Poincar\'e inequality.

\begin{lem}\label{L.SP}
Assume that $\Omega$ is a bounded,  connected  open set 
in $\R^n$ with Lipschitz boundary.
Then there exists a constant $C = C(\Omega)$ such that if
$u\in BV(\Omega)$ is a function such that
\[
\mathcal L^n( \mbox{supp}(u)) \le \frac 12 
\mathcal L^n(\Omega) .
\]
then 
\begin{equation}\label{sp.ineq}
\int_\Omega |u|^{\frac n{n-1}} dx \le C \left( \int_\Omega |Du| dx \right)^{\frac n{n-1}}
\end{equation}
\end{lem}

\begin{proof}
This is proved in for example in \cite[Theorem 3.51]{afp}. The proof there
assumes that $\Omega$ is a  ball, but the argument only requires that
the relative isoperimetric inequality hold on $\Omega$, {\em i.e.} that there exist some $C = C(\Omega)$ such that
$
(\mathcal L^n(E ))^{\frac{n-1}n} \le C\  \mbox{Per}_{\Omega}(E)
$
for every $E\subset \Omega$ of finite perimeter such that 
$\mathcal L^n(E ) \le \frac 12 \mathcal L^n(\Omega)$.
This is known to hold for bounded connected Lipschitz domains, see for example 
\cite[4.5.2]{Federer}
so the proof  of \eqref{sp.ineq}  in \cite{afp} applies here.
\end{proof}

\begin{proof}
[Proof of Theorem \ref{main.theorem}]

We must estimate the $H^1_\eps$ norm
of $u_\eps - U_\eps$. We will consider separately the region $\mathcal N'$ near $\Gamma$, where we can use normal coordinates, and its complement $\mathcal M$.

\paragraph{1. Estimates in normal coordinates near $\Gamma$.}

Recall that $V_\eps(y_0, y_1, y_2) = \tau_{s_*(y_0,y_1)}Q_\eps(y_2) = U_\eps \circ \Psi$.
We will first prove that
\begin{equation}\label{v-V}
\int_{\mathbb S^1\times I} \frac 1 \eps (v_\eps -V_\eps)^2 +
\sum_{i=0}^2 \eps[\partial_{y_i}(v_\eps -V_\eps)]^2 dy_1 dy_2 \lesssim \eps^2 \quad
\mbox{ for all }y_0\in (-T_1,T^1).
\end{equation}
To start, using Lemma \ref{L.Linf} implies that hypothesis \eqref{L2c3} of Proposition \ref{proposition1}  is satisfied for every 
$(y_0, y_1)\in (-T_1,T^1)\times\mathbb S^1$.
The proposition then yields
\begin{align*}
&\int_I \frac 1\eps (v_\eps (y_0,y_1, y_2) - V_\eps(y_0,y_1,y_2))^2 dy_2
\\
&\hspace{2em}
+ \eps\int_I\Big( (\partial_{y_2}(v_\eps (y_0,y_1, y_2) - V_\eps(y_0,y_1,y_2))\Big)^2dy_2
\lesssim \theta_1( v_\eps(y_0, y_1, \cdot)) + C e^{-c/\eps}
\end{align*}
for every $(y_0,y_1)$.
Integrating over $\mathbb S^1$ and using \eqref{Th.th} and \eqref{Thetas.est},
we find that
\[
\int_{\mathbb S^1\times I} \frac 1 \eps (v_\eps - V_\eps)^2 + \eps(\partial_{y_2}(v_\eps - V_\eps))^2 \ dy_2\, dy_1 \le C \eps^2
\]
for every $y_0\in (-T_1,T^1)$.
It also follows from \eqref{Thetas.est} that 
\begin{equation}\label{tangential.H1}
\int_{\mathbb S^1\times I} \eps(\partial_{y_0}v_\eps )^2 +
\eps(\partial_{y_1}v_\eps)^2 \ dy_2\, dy_1 \le C \eps^2
\end{equation}
for every $y_0$ as above. Thus, to complete the proof of \eqref{v-V}, it suffices to show
that $V_\eps$ satisfies an estimate similar to \eqref{tangential.H1}.
For this, we compute
\begin{equation}\label{pyjVe}
\partial_{y_j} V_\eps =   -(\tau_{s_*}Q_\eps' ) \, \partial_{y_j}s_*\qquad
\qquad\mbox{ for $j=0,1$}.
\end{equation}
Also, writing $W_\eps := v_\eps - V_\eps$ and differentiating,
\begin{equation}\label{pyjve}
( \tau_{s_*}Q_\eps' ) \, \partial_{y_j}s_*
= 
\partial_{y_j}W_\eps -
\partial_{y_j}v_\eps .
\end{equation}
We want to multiply both sides of this identity by 
$\partial_{y_j}V_\eps= -(\tau_{s_*}Q_\eps' ) \, \partial_{y_j}s_*$ 
and integrate. In order to simplify the term involving $\partial_{y_j}W_\eps$,
we recall from Proposition \ref{prop1} that 
\[
\int_I W_\eps(y_0,y_1, y_2) \, Q_\eps'(y_2 -  s_*(y_0,y_1)) \ dy_2= 0 \qquad
\mbox{ for all }(y_0,y_1).
\]
Differentiating with respect to $y_j$ for $j=0,1$
yields
\[
\int_I (\partial_{y_j}W_\eps) (\tau_{s_*}Q_\eps') \,dy_2 \ = 
\int_I W_\eps\  (\tau_{s_*}Q_\eps'' )(\partial_{y_j}s_*) \ dy_2
\qquad\mbox{ for all }(y_0,y_1).
\]
Since $s_*$ is independent of $y_2$, 
it follows that
\begin{align*}
\int_I (\partial_{y_j}W_\eps) (\tau_{s_*}Q_\eps')(\partial_{y_j}s_*) \,dy_2 \ 
&\le
(\partial_{y_j}s_*)^2 
\| W_\eps \|_{L^2(I)} \| \tau_{s_*}Q_\eps''\|_{L^2(I)} \\
&\overset{\eqref{L2again}}\lesssim
(\partial_{y_j}s_*)^2 
\eps^{3/4} \eps^{-3/2} = \eps^{-3/4}(\partial_{y_j}s_*)^2 
\end{align*}
for all $(y_0,y_1)$.

If we multiply \eqref{pyjve} by $(\tau_{s_*}Q_\eps')\partial_{y_j}s_*$
and integrate first with respect to $y_2$, then with respect to $y_1$,
we therefore deduce that
\begin{equation}\label{gettingthere}
\int_{\mathbb S^1} (\partial_{y_j}s_*)^2\int_{I} (\tau_{s_*}Q_\eps')^2 dy_2\, dy_1 \lesssim
\eps^{-3/4}\int_{\mathbb S^1}(\partial_{y_j}s_*)^2 dy_1
+
\int_{\mathbb S^1\times I}
\partial_{y_j}v_\eps^2 \, dy_2\, dy_1,
\end{equation}
where we have used the elementary estimate
\[
\int_{\mathbb S^1\times I}
(\partial_{y_j}v_\eps) (\partial_{y_j}s_*)( \tau_{s_*}Q_\eps') \, dy_2\, dy_1
\le
\frac 12 \int_{\mathbb S^1\times I}
(\partial_{y_j}v_\eps^2) + (\partial_{y_j}s_* )^2(\tau_{s_*}Q_\eps')^2 \, dy_2\, dy_1,
\]
In addition,
\[
\int_I (\tau_{s_*}Q_\eps')^2 dy_2 = \frac {c_0}\eps + O(e^{-c/\eps})
\ge \frac{c_0}{2\eps}\qquad\mbox{ for every }(y_0,y_1)
\]
for $\eps$ small enough. Thus the first term on the right-hand side of \eqref{gettingthere} can be absorbed by the left-hand side , and we can  finally conclude that
\[
\int_{\mathbb S^1} (\partial_{y_j}s_*)^2 dy_1 \lesssim \eps \int_{\mathbb S^1\times I}
(\partial_{y_j}v_\eps)^2 \lesssim \eps^2 \qquad\mbox{ for all }y_0\in (-T_1, T^1).
\]
With this, we readily deduce from \eqref{pyjVe} that $\| \partial_{y_j} V_\eps\|_{L^2(\mathbb S^1\times I)}^2 \le C \eps$ for $j=0,1$, completing the proof
of \eqref{v-V}.

Also, for every $y_0$, we know from \eqref{sstar} that 
\[
\int_{\mathbb S^1} s_*^2(y_0, y_1) dy_1 \lesssim \int_{\mathbb S^1} (\eps^2 + \theta_1(v_\eps(y_0, y_1, \cdot) )dy_1 \lesssim \eps^2 + \Theta_1(y_0) \lesssim \eps^2,
\]
so we have  proved that $\| s_* (y_0, \cdot)\|_{H^1(\mathbb S^1)} \le C\eps$
for every $y_0$, which is \eqref{sstar.H1}.

\paragraph{2. Estimates in $(t,x)$ coordinates near $\Gamma$.}

Since $(u_\eps - U_\eps) = (v_\eps - V_\eps)\circ \Psi^{-1}$, and because $\Psi$
is a diffeomorphism from $(-T_1, T^1)\times \mathbb S^1\times I$ onto
its image, which contains $\mathcal {N}'$, a simple change of variables
shows that \eqref{v-V} implies that
\[
\| u_\eps - U_\eps\|_{H^1_\eps(\mathcal {N}')} \le C\eps.
\]

\paragraph{3. Estimates in $(t,x)$ coordinates away $\Gamma$.}

To finish the proof, we must estimate $\|u_\eps - U_\eps\|_{H^1_\eps(\mathcal M)}$.

Note that $\mathcal M$ consists of two components,
$\mathcal M \cap \mathcal O$ and $\mathcal M \setminus \mathcal O$,
with $U_\eps = 1$ in the former and $U_\eps = -1$ in the latter. (Recall
that $\mathcal O$ is the region enclosed by $\Gamma$.)

We already know from Proposition \ref{old.results} that $\| D u_\eps \|_{L^2(\mathcal M)}^2
\le C\eps$, and since $Du_\eps = D(u_\eps - U_\eps)$ in $\mathcal M$, it only remains to
prove that
\begin{equation}\label{u-U}
\int_{\mathcal M \cap \mathcal O} (u_\eps-1)^2 +
\int_{\mathcal M \setminus \mathcal O} (u_\eps+1)^2 \le C\eps^3.
\end{equation}
As we will see, these are straightforward consequences of results from \cite{jer0909}.
We first consider $\mathcal M\cap \mathcal O$.

We will write
\[
H(s) := (\frac  13 s^3 - s)^+ =
 \begin{cases}|s - \frac 13 s^3| &\mbox{if } \  \ - \sqrt 3 \le s\le 0, \\
 0&\mbox{otherwise} \ .
\end{cases}
\]
It is easy to see that 
\[
(u_\eps - 1)^2 \lesssim (u_\eps^2-1)^2 + H(u_\eps)^{3/2}.
\]
We already know from \eqref{uep.far} that
$
\int_{\mathcal M\cap \mathcal O}(u_\eps^2-1)^2 \ dx\ dt \lesssim \eps^3,
$
so  to prove \eqref{u-U}, it suffices to show that 
\[
\int_{\mathcal M\cap \mathcal O}H(u_\eps)^{3/2}dx\ dt \lesssim \eps^3.
\]
In doing so, we will use the fact (which motivates the definition of $H$) that
\[
|DH(u_\eps)| \le |u_\eps^2-1| \, |Du_\eps|  \le \frac \eps 2 |Du_\eps|^2 + \frac 1{2\eps}(u_\eps^2-1)^2.
\]
As a result
\[
\int_{\mathcal{M\cap O}} |DH(u_\eps)|\, dx\, dt \lesssim  \eps^2
\]
by Proposition \ref{old.results}.
Therefore, to complete the proof of \eqref{u-U}, it is enough to
note that
\[
\int_{\mathcal{M\cap O}} |H(u_\eps)|^{3/2}\, dx\, dt 
\lesssim 
\left(\int_{\mathcal{M\cap O}} |DH(u_\eps)|\, dx\, dt \right)^{3/2} \lesssim \eps^3.
\]
But this follows from the Sobolev-Poincar\'e estimate in Lemma \ref{L.SP}.
The estimate is applicable here since 
\eqref{u_weakly_close_toA} implies that $\{ (t,x)\in \mathcal M\cap \mathcal O : 
u_\eps(t,x) < 0 \}$ has measure at most $C\eps^{1/2} \le \mathcal L^n(\Omega)$ for $\eps$ small. The same thus holds
for $\{ (t,x)\in \mathcal M\cap \mathcal O  : H(u_\eps)(t,x) \ne 0\}$,
and this is the hypothesis for Lemma \ref{L.SP}.

The argument for $\mathcal M\setminus \mathcal O$
is almost identical. The only point to notice is that, since 
$u_\eps = \sign_{\mathcal O}$ on $-T_*,T^*)\times (\R^2\setminus B(R))$
see \eqref{comp_supp}, it suffices to 
show that 
\[
\int_{(t,x)\in \mathcal M \setminus \mathcal O : |x|< R} (u_\eps+1)^2 \le C\eps^3.
\]
Since the domain of integration here is a bounded connected Lipschitz set whenever $R$ is large enough, Lemma \ref{L.SP} applies, and we may now argue exactly as above.
\end{proof}

\paragraph{Acknowledgement.}
The research of both authors was partially supported by the Natural Sciences and
Engineering Research Council of Canada under operating grants 2017-04313 (El Smaily) and 261955 (Jerrard).

\end{document}